\documentclass[ngerman,american]{article}
\usepackage[T1]{fontenc}
\usepackage[utf8]{luainputenc}
\usepackage{amsthm}
\usepackage{amsmath}
\usepackage{amssymb}

\makeatletter

\newcommand{\noun}[1]{\textsc{#1}}

\theoremstyle{plain}
\newtheorem{thm}{\protect\theoremname}
  \theoremstyle{plain}
  \newtheorem{lem}[thm]{\protect\lemmaname}
  \theoremstyle{plain}
  \newtheorem{cor}[thm]{\protect\corollaryname}
  \theoremstyle{plain}
  \newtheorem{algorithm}[thm]{\protect\algorithmname}
  \theoremstyle{remark}
  \newtheorem*{acknowledgement*}{\protect\acknowledgementname}

\makeatother

\usepackage{babel}
  \addto\captionsamerican{\renewcommand{\acknowledgementname}{Acknowledgement}}
  \addto\captionsamerican{\renewcommand{\algorithmname}{Algorithm}}
  \addto\captionsamerican{\renewcommand{\corollaryname}{Corollary}}
  \addto\captionsamerican{\renewcommand{\lemmaname}{Lemma}}
  \addto\captionsamerican{\renewcommand{\theoremname}{Theorem}}
  \addto\captionsngerman{\renewcommand{\acknowledgementname}{Danksagung}}
  \addto\captionsngerman{\renewcommand{\algorithmname}{Algorithmus}}
  \addto\captionsngerman{\renewcommand{\corollaryname}{Korollar}}
  \addto\captionsngerman{\renewcommand{\lemmaname}{Lemma}}
  \addto\captionsngerman{\renewcommand{\theoremname}{Theorem}}
  \providecommand{\acknowledgementname}{Acknowledgement}
  \providecommand{\algorithmname}{Algorithm}
  \providecommand{\corollaryname}{Corollary}
  \providecommand{\lemmaname}{Lemma}
\providecommand{\theoremname}{Theorem}

\begin{document}

\title{An Algorithmic Approach to Pick's Theorem}

\author{Haim Shraga Rosner}
\maketitle
\begin{abstract}
We give an algorithmic proof of Pick's theorem which calculates the
area of a lattice-polygon in terms of the lattice-points.
\end{abstract}

\paragraph{Introduction}

Pick's Theorem was first discovered by Georg A\@. Pick in 1899 \cite{Pick}.
Many different proofs for this elegant theorem have been published
over the last 60 years. Some found a topological connection with Euler's
formula, and others, like Pick himself, proved it by geometrical means.
Most of the geometric proofs prove the additivity of Pick's formula
and find a specific example for which this formula gives the area.
Both Liu and Varberg \cite{key-19,Varberg} mentioned that the most
challenging part of some proofs is the fact that a primitive lattice-triangle
is of area $\frac{1}{2}$. Varberg, for example, bypasses that fact
in his proof. Here we do use this fact and find an explicit algorithm
to find all lattice points for a lattice-polygon $P$.
\begin{thm}
\emph{\label{thm:Pick}(Pick, \cite{Pick})}. Let $P$ be a lattice-polygon.
Then its area is $i+\frac{u}{2}-1$, when $i$ is the number of interior
lattice-points of $P$ and $u$ is the number of its boundary lattice-points.\end{thm}
\begin{lem}
\label{lem:minimal-area}The minimal possible area of a triangle whose
vertices are all lattice-points is~$\frac{1}{2}$.\end{lem}
\begin{proof}
\cite{key-19} Let $A,B,C$ be lattice-points, and denote $\triangle ABC$
by $T$. We can bind $T$ with a rectangle parallel to the axes. In
order to calculate the area of $T$, subtract the area outside it
from the rectangle. That area consists of several right triangles
and may also include a rectangle. The area of each right triangle
is half the product of its legs, which are natural numbers, and thus
is a multiple of $\frac{1}{2}$. The area of both the big and the
small rectangles are natural numbers. Therefore, the area of $T$
is a multiple of $\frac{1}{2}$. We have found that the minimum positive
area of $T$ is~$\frac{1}{2}$.

\noindent \emph{Alternative Proof}. \cite{key-13,Honsberger,EulerII,Minkow}
Let $A,B,C$ be lattice-points. The area of $\triangle ABC$ is $\left.\frac{1}{2}\cdot\left|\det\begin{pmatrix}x_{A} & x_{B} & x_{C}\\
y_{A} & y_{B} & y_{C}\\
1 & 1 & 1
\end{pmatrix}\right|\right.$, which is an integer multiple of~$\frac{1}{2}$.\end{proof}
\begin{thm}
\label{thm:minimal iff half}Let $A,B,C$ be lattice-points. The triangle
$\triangle ABC$ is of minimal area, i.e., $\frac{1}{2}$, iff $\triangle ABC\cap\mathbb{Z}^{2}=\left\{ A,B,C\right\} $,
i\@.e\@., iff there are no lattice-points on the edges of $\triangle ABC$,
nor in the interior of $\triangle ABC$.\end{thm}
\begin{proof}
$\left(\Longrightarrow\right)$. We prove the contrapositive. Assume
that there is another lattice-point $D$ in that intersection. If
$D$ is an interior point, then we can decompose the triangle into
three triangles: $\triangle ABD$, $\triangle ACD$ and $\triangle BCD$.
If $D$ is on an edge of $\triangle ABC$, then we can decompose it
into two triangles by drawing a line between $D$ and the opposite
vertex. In any case, $\triangle ABC$ contains several disjoint sub-triangles,
and thus its area is at least twice as much as the minimum. We conclude
that if $\triangle ABC$ has such a point $D$, then its area is not
minimal.

$\left(\Longleftarrow\right)$. Denote $\triangle ABC$ by $T$. Move
the point $C$ to the origin $\left(0,0\right)$, and denote the other
points as $A=\left(a,c\right)$ and $B=\left(b,d\right)$. The area
of $T$ is $\frac{\left|ad-bc\right|}{2}$. Since $T$ is a triangle,
its area is non-zero. W.l.o.g. we assume that $ad-bc>0$ and $c\le d$.
Denoting $n=ad-bc$, we want to prove that if $n>1$ then there is
another lattice-point $D$ in $T$. If $\gcd\left(a-b,c-d\right)=k>1$,
then the point $\frac{k-1}{k}A+\frac{1}{k}B$ is a new lattice-point
on the edge $AB$, as desired. Thus we assume that $\gcd\left(a-b,c-d\right)=1$.
We prove that such a lattice-point exists on the segment $\frac{n-1}{n}AB$.

The equation of this segment is $\left(a-b\right)y-\left(c-d\right)x=n-1$.
Since $a-b$ and $c-d$ are coprime, there exist $s$ and $t$ such
that $\left(a-b\right)s-\left(c-d\right)t=1$. We multiply this equation
by $n-1$ and get $\left(a-b\right)\left(n-1\right)s-\left(c-d\right)\left(n-1\right)t=n-1$.
Thus, we take $x=\left(n-1\right)t$ and $y=\left(n-1\right)s$, to
find a lattice-point on that line. However, we need to find a lattice
point\emph{ }not only on that \emph{line} but on the \emph{segment}
$\frac{n-1}{n}AB$. Thus, we replace $x$ by $x+\left(a-b\right)i$
and $y$ by $y+\left(c-d\right)i$, to get a new lattice-point on
the line. For all $r\in\mathbb{R},$ we can choose an appropriate
$i$ such that $r\le y<r-\left(c-d\right)$. We choose the appropriate
$i$ for $r=\frac{n-1}{n}c$, i.e., $c-\frac{c}{n}\le y<d-\frac{c}{n}$.
Denote this point by $D=\left(x,y\right)$. We claim that this $D$
is in the segment, i.e., $\frac{n-1}{n}c\le y\le\frac{n-1}{n}d$,
and consequently, $\frac{n-1}{n}a\le x\le\frac{n-1}{n}b$ as well
(or $\frac{n-1}{n}a\ge x\ge\frac{n-1}{n}b$, if $a\ge b$). We need
to demonstrate that $D$ does not fall past $\frac{n-1}{n}B$, i.e.,
that $y\notin\left(d-\frac{d}{n},d-\frac{c}{n}\right)$. If $D$ were
past $B$, then $nD$ would be a lattice-point on the line $\left(n-1\right)AB$,
past the lattice-point $\left(n-1\right)B$. But the $y$-difference
between these two points would be $ny-\left(n-1\right)d$. In accordance
with $d-\frac{d}{n}<y<d-\frac{c}{n}$, we find that $0<ny-\left(n-1\right)d<d-c$.
This $y$-difference between lattice-points on a line with slope $\frac{c-d}{a-b}$
contradicts the fact that $a-b$ and $c-d$ are coprime. In conclusion,
$D$ is in the segment $\frac{n-1}{n}AB$.

We have found another lattice-point $D$ in the triangle $T$ of area
greater than~$\frac{1}{2}$.

\end{proof}
The above theorem is of course equivalent to \cite[Theorem 34]{HardyWright},
which overlooked parallelograms instead of triangles, although they
neglected the case in which there are two (or more) points on the
diagonal $PQ$. Moreover, we want to mention this connection as evidence
to the deep connection between Pick's theorem and Farey series. We
use some concepts that appear there in §3\@.4-3\@.7.
\begin{cor}
For a lattice-triangle $\triangle ABC$, with $A=\left(a,c\right)$,
$B=\left(b,d\right)$ and $C=\left(0,0\right)$, if $a-b$ and $c-d$
are coprime, than there is one lattice point in 
\begin{equation}
X=\left\{ \frac{n-i}{n}A+\frac{i-1}{n}B\colon i=1,\ldots,n\right\} ,\label{eq:SegPts}
\end{equation}
for $n=\left|ad-bc\right|$.\end{cor}
\begin{proof}
$X$ is a subset of the segment $\frac{n-1}{n}AB$, and therefore
$nX$ is a subset of the segment $\left(n-1\right)AB$. We find that
$nX$ is the set of the $n$ lattice-points on $\left(n-1\right)AB$.
Multiplying the point $D$ from the end of the proof of theorem \ref{thm:minimal iff half}
by~$n$ gives a lattice-point on the segment $\left(n-1\right)AB$.
We have shown that $nD\in nX$, and thus $D\in X$.
\end{proof}
This proof provides an explicit way of finding the lattice-points
of a lattice-polygon.
\begin{algorithm}
(Lattice-triangulation of a lattice-polygon $P$) Let $P$ be a lattice
polygon. We want to partition $P$ into minimal triangles. We make
a list of these triangles via the following steps:
\begin{enumerate}
\item Partition $P$ into triangles, by drawing lines between non-adjacent
vertices, without crossing any other line.
\item Choose one triangle, $T$. Choose one vertex of $T$, which we specify
as $C$.
\item Move $C$ to the origin, and denote the other vertices as $A=\left(a,c\right)$
and $B=\left(b,d\right)$.
\item If $a-c$ and $b-d$ are not coprime, then take $D=\frac{k-1}{k}A+\frac{1}{k}B$
for $k=\gcd\left(a-c,b-d\right)$. Partition $T$ into $T_{1}=\triangle ACD$
and $T_{2}=\triangle BCD$. Return both $T_{1}$ and $T_{2}$ to step
2.
\item If $a-c$ and $b-d$ are coprime and the area of $T$ is $\frac{n}{2}$
with $n>1$, take $D$ the one and only lattice-point in (\ref{eq:SegPts}).
Partition $T$ into $T_{1}=\triangle ABD$, $T_{2}=\triangle ACD$
and $T_{3}=\triangle BCD$. Return $T_{1}$, $T_{2}$ and $T_{3}$
to step 2.
\item If the area of $T$ is $\frac{1}{2}$, $T$ is minimal,and we add
$T$ to our list. If there are other triangles with area greater than
$\frac{1}{2}$, return them to step 2.
\item We have obtained a list of minimal triangles. This procedure must
terminate, since the number of such triangles is twice the area of
$P$.
\end{enumerate}
\end{algorithm}
We conclude the proof of Pick's Theorem (Theorem \ref{thm:Pick})
by proving that Pick's formula is additive.
\begin{proof}
\cite{Pick} Let $P$ be a lattice-polygon, and denote by $i$ and
$u$ the number of its interior points and boundary points, respectively.
We claim that Pick's formula, $i+\frac{u}{2}-1$, is additive under
triangulation, like the total area. Thus, we can triangulate $P$
into minimal triangles, and calculate that for a minimal triangle
$0+\frac{3}{2}-1=\frac{1}{2}$, and conclude the proof of Pick's Theorem.

If $i\neq0$, choose an interior point, $D$, and two boundary points
$A,B$. Partition the polygon $X$ into two polygons by drawing the
lines $AD$ and $BD$. Denote by $u_{1}$ and $u_{2}$ the number
of boundary points in these two polygons, and by $i_{1}$ and $i_{2}$
the number of their respective interior points. Denote by $d$ the
total number of lattice points on the segments $AD$ and $BD$ (count
$A$, $B$ and $D$ only once!). Clearly, $i=i_{1}+i_{2}+d-2$, since
$d$ counts the points $A$ and $B$, which are not interior points.
Furthermore, $u=u_{1}+u_{2}-2d+2$. We subtracted here the points
on the segments from both polygons, but added the points $A$ and
$B$. Thus,
\begin{eqnarray}
i+\frac{u}{2}-1 & = & i_{1}+i_{2}+d-2+\frac{u_{1}+u_{2}-2d+2}{2}-1\nonumber \\
 & = & \left(i_{1}+\frac{u_{1}}{2}-1\right)+\left(i_{2}+\frac{u_{2}}{2}-1\right).\label{eq:additivity}
\end{eqnarray}
We conclude that $i+\frac{u}{2}-1$ is preserved when partitioning
$P$ with respect to an interior point $D$.

If $i=0$, but $u>3$, choose two boundary points, $A$ and $B$,
and take $D$ to be the same as $A$. Equation (\ref{eq:additivity})
is true in this case as well.

If $i=0,u=3$, we can no longer partition that minimal triangle, but
by Theorem \ref{thm:minimal iff half}, we find that the area of this
polygon is $\frac{1}{2}=0+\frac{3}{2}-1=i+\frac{u}{2}-1$.

In conclusion, by decomposing $P$ by any point, interior or boundary,
the total area is the sum of the area of the two parts, and Pick's
formula for $P$ is the sum of Pick's formulas for them. Hence, we
find that both these quantities are additive. If these two quantities
coincide on minimal triangles, then by induction they coincide on
any lattice-polygon. Indeed, by Theorem \ref{thm:minimal iff half},
the area of each minimal triangle $T$ is $\frac{1}{2}=0+\frac{3}{2}-1=i_{T}+\frac{u_{T}}{2}-1$.
From additivity of both this formula and the concept of area, $i+\frac{u}{2}-1$
is the area of $P$.\end{proof}
\begin{acknowledgement*}
The author would like to thank Louis Rowen for his help, patience
and guidance.
\end{acknowledgement*}

\end{document}